\documentclass[10pt,twoside]{article}
\usepackage[all]{xy}
\usepackage{amsfonts,amsmath,oldgerm,amssymb}
\usepackage{theorem}
\newcounter{elno}                

\newtheorem{thm}{Theorem}[section]
\newtheorem{thm*}{Theorem}
\newtheorem{defn}[thm]{Definition}
\newtheorem{prop}[thm]{Proposition}
\newtheorem{lem}[thm]{Lemma}
\newtheorem{cor}[thm]{Corollary}
\newtheorem{notation}[thm]{Notation}

\newtheorem{thmm*}{Theorem}
\newtheorem{blemma}{Basic Lemma -- first form}

\newtheorem{bslem}{Basic Lemma -- second form}

\newtheorem{nprop}{Proposition 1.3A}

\newtheorem{ncor}{Corollary 1.3B}

\newtheorem{nncor}{Corollary 1.3C}

\newcommand{\Z}{{\mathbb {Z}}}

\newcommand{\homology}{{\mathrm {H}}}
\newcommand{\C}{{\mathbb {C}}}
\newcommand{\R}{{\mathbb {R}}}
\newcommand{\V}{{\mathrm {V}}}
\newcommand{\A}{{\mathbb {A}}}
\newcommand{\Hom}{{\operatorname{Hom}}}
\newcommand{\Ext}{{\operatorname{Ext}}}
\newcommand{\HHom}{{\mathbf{Hom}}}
\newcommand{\EExt}{{\mathbf{Ext}}}
\newcommand{\derived}{{\operatorname{R}}}
\newcommand{\coker}{\operatorname{coker}}

\newcommand{\qed}{\nopagebreak\par\hspace*{\fill}$\square$\par\vskip2mm}

\newcommand{\myauthor}[1]{\begin{center}%
{\Large #1}\bigskip
\end{center}}

\newcommand{\myaddress}[1]{\indent {\sc #1}\par}
\newcommand{\myemail}[1]{\indent \emph{E-mail:} {\tt #1}}

\theorembodyfont{\upshape}
\newtheorem{proof}{Proof}

\newtheorem{remark}[thm]{Remark}

\begin{document}

\markboth{Madhav V. Nori}{Constructible Sheaves}
\markright{Constructible Sheaves}

\thispagestyle{plain}
\begin{center}
{\LARGE\bf Constructible Sheaves}
\end{center}

\myauthor{Madhav V. Nori}

\begin{abstract}
 This article contains a proof of the basic lemma, which
yields a motivic proof of the Andreotti Frankel theorem 
for affine varieties. Next, it is shown that
the triangulated category of ``Cohomologically
Constructible Sheaves'' (as it is referred to in the Riemann-Hilbert
correspondence) coincides with the derived category of bounded 
complexes of constructible sheaves. It is also shown that higher
direct images and the sheaf-Ext groups are effaceable in
the category of constructible sheaves. 
\end{abstract}

\section*{Introduction}

     The Andreotti-Frankel theorem asserts that a 
closed complex manifold $X$ of $\C ^N$ of dimension $n$ has the
homotopy type of a cell complex of dimension at most $n$ (see \cite{nori:AF}
or \cite{nori:Mil}). This is achieved by constructing a Morse function 
$f:X\to \R$ with critical points of index at most $n$. Put 
$$
X_a=\{ x\in X: f(x)\leq a\}.
$$ 
The cohomology of the
pair $(X_a,X_b)$ when $b<a$ now enters the picture. But
 the cohomology of these pairs do
not inherit the rich structure (Galois action , mixed Hodge structure,
for instance) that the cohomology of $X$ carries when $X$ is an 
algebraic variety. This is the reason for 
formulating the basic lemma below.

\begin{blemma}[A. Beilinson] 
Let $k$ be a subfield of $\C$. 
Let $W$ be Zariski closed in an affine variety
$X$ defined over $k$. Assume $\dim(W)\!<\!\dim(X)$. 
Then there is a Zariski 
closed $Z$ in $X$ so that $\dim(Z)\!<\nolinebreak\! n$ with $W\subset Z$, 
and $\homology^q(X,Z)=0$ whenever $q\neq \dim (X)$.
\end{blemma}

\medskip

In the above lemma, $\homology^q(X,Z)$ denotes the singular
cohomology of the pair $(X(\C),Z(\C))$.
The cohomological version of the Andreotti-Frankel theorem for affine varieties 
is of course an immediate consequence. 
Indeed if $X$ is affine of dimension $n$, from the above lemma, we 
deduce an increasing sequence of Zariski closed sets $X_i$ of dimension
at most $i$ so that $\homology^j(X_i, X_{i-1})=0$ whenever 
$j\neq i$.  Consequently, the $j$-th cohomology 
of the complex $D^{\bullet}$, where 
$D^i=\homology ^i(X_i,X_{i-1})$ coincides with the $j$-th cohomology of $X$,
which therefore vanishes whenever $j>n$. Furthermore   $D^{\bullet}$
may be regarded as a complex of ``motives''. See also 
\cite{nori:Be1} and \cite{nori:Be2} for such statements for mixed Hodge structures  
and Galois representations.
  
    The proof given here of the basic lemma is geometric.  
It is not hard to see that for our choice
of $Z$ in the basic lemma, $(X(\C),Z(\C))$ is in fact,
upto homotopy, a relative CW pair obtained by
attaching cells of dimension $n$. Furthermore, the same method
gives a triangulation of real affine semi-algebraic sets. 

Our proof is valid only in characteristic zero. The proof of
the intermediate Proposition~1.3 is not valid in positive
characteristic due to the presence of wild ramification.

The only proof of the basic lemma in positive characteristic
is that of Beilinson (\cite[Lemma~3.3]{nori:Be2}). In fact, Beilinson's  
 Lemma~3.3 is a vast generalisation of the basic lemma
 to perverse sheaves, in the sense that it
produces plenty of perverse sheaves with at most one 
nonvanishing hypercohomology. With $X$ and $W$ as in the basic lemma,  
the proof of Lemma~3.3, \cite{nori:Be2}, as Beilinson explained to the author,
yields $Z$ in the following shape: first enlargen $W$ so
that its complement $V=X-W$ is affine and smooth of pure 
dimension $n$, where $n=\dim (X)$. Next, embed 
$X$ as a Zariski locally closed set in projective space, 
and let $H$ be a general hyperplane section of $X$.
Then $\homology^q(X,W \cup H)=0$ for $q \neq n$, and thus we 
may take $Z=W \cup H$ in the basic lemma. Beilinson's Lemma~3.3 
also includes a very general
cohomological (rather than homotopy theoretic) version
 of the ``Lefschetz hyperplane section theorem for complements'':
 with $V=X-W$ and $H$ as above, his result shows that 
$\homology^q(V,V \cap H)=0$ for $q \neq n$. 
 However, Beilinson relies on M. Artin's
 sheaf-theoretic version of the Andreotti-Frankel theorem,
 whereas we deduce M. Artin's theorem (in characteristic
zero).  

  We now turn to constructible sheaves. A  
ring $R$ will be
fixed once and for all. All sheaves considered are sheaves of 
left $R$-modules.  
A subfield $k$ of the complex numbers will remain fixed throughout the 
paper. All varieties and morphisms considered are defined over this field
$k$. By a  ``sheaf on $X$'' we mean a sheaf of $R$-modules on the set
$X(\C)$ of
$\C$-rational points equipped with the usual topology. A sheaf $F$
on a variety $X$ is said to be \emph{weakly constructible}  
if $X$ is the disjoint union of a
finite collection of locally closed subschemes $Y_i$ so that 
the restrictions $F|Y_i$ are all locally constant sheaves.
We also fix a full Abelian subcategory $\mathcal{N}$ of the category of 
all $R$-modules so that every $R$-module $M$ that is isomorphic to
an object of $\mathcal{N}$ is actually an object of $\mathcal{N}$.
We call $F$ \emph{constructible} if, in addition, 
all the stalks of $F$ are objects of $\mathcal{N}$. 
We warn the reader that, in the literature (e.g., \cite[Chapter~VIII]{nori:KS}),
when constructible sheaves are discussed, 
$\mathcal{N}$ is the category of all Noetherian modules.
But we do not place such restrictions on 
  $\mathcal{N}$.

For constructible or weakly constructible sheaves $F$, the sheaf
cohomology   $\homology^q(X(\C),F)$ will be denoted simply by 
$\homology^q(X,F)$. To compute sheaf cohomology, one
takes a resolution  $K^{\bullet}$ of the given sheaf $F$
and considers 
the cohomology of the complex: $\Gamma (X,K^{\bullet})$. 
If $K^{\bullet}$
is an injective, or even a flasque, resolution (e.g., the Godement
resolution), we get:
$$
\homology^q (\Gamma (X,K^{\bullet}))=\homology^q(X,F).
$$

Constructible sheaves are rarely flasque (if so, they are supported
on a finite set). Nevertheless, it is true that every
constructible sheaf $F$ on $\A^n$ has a \emph{constructible} resolution
 $K^{\bullet}$ so that 
$\homology^q( \Gamma (\A^n,K^{\bullet}))=\homology^q(\A^n,F)$ 
 for all $q \geq 0$. Furthermore, we may assume that
$K^q=0$ for all $q>n$. This is a consequence
of Theorem~1 below.  The comparison theorem of Artin-Grothendieck
for etale cohomology (see \cite[Expose~XI]{nori:SGA4}) is the special case of
Theorem~2 for $R=\Z /n\Z$. 

\begin{thm*} Every constructible  sheaf $F$ on affine $n$-space $\A^n$ 
is a subsheaf of a constructible  $G$ so that $\homology^q(\A^n,G)=0$ for all 
$q>0 $.
\end{thm*}

\begin{thm*} For every constructible sheaf $F$ on a variety
$X$, there is a monomorphism $a:F\to G$ with $G$ constructible  so that 
$$
\homology^q(X,a):\homology^q(X,F) \to \homology^q(X,G)
$$ 
is zero, for all $q>0$.
\end{thm*}
 
Further effaceability results and their consequences are
formulated below.  We shall denote by $Sh(X)$ the category of sheaves of 
$R$-modules on $X$ and by $C(X)$ the full subcategory of $Sh(X)$ with 
Obj $C(X)$= constructible sheaves. 
An additive functor $T:C(X)\to A$, where $A$ is an Abelian category
is effaceable, if for all objects $F$ of $C(X)$, there is a 
monomorphism $a:F\to G$ in $C(X)$ so that $Ta:TF\to TG$ is
the zero morphism.  

\begin{thm*} \begin{enumerate}
\item[{\rm (a)}] Assume that $R$ is commutative.   
For every weakly constructible sheaf 
$P$ on $X$ such that all the stalks of $P$ are finitely
generated projective modules, and for every $q>0$, 
the functor 
$$
\Ext_{Sh(X)}^q(P, \cdot)|C(X)
$$ 
is effaceable. 
\item[{\rm (b)}] Let $R$ be a field, and let $\mathcal{N}$  be the
category of all finite dimensional vector spaces over $R$.
 Let $F,G$ be constructible sheaves on $X$. Then 
 $\Ext^q_{C(X)}(F,G) \to \Ext^q_{Sh(X)}(F,G)$ is an 
isomorphism for all $q \geq 0$.
\end{enumerate}
\end{thm*}

Throughout the paper, $\Ext^q_{\mathcal{A}}(F,G)$ denotes the Yoneda Ext group 
for objects $F$ and $G$ of an Abelian category $\mathcal{A}$.
That Theorem~3(a) implies Theorem~3(b) is the
standard application of effaceability (see 
\cite[Prop.~2.2.1,~page~141]{nori:Toh} ). Theorem~3(b) implies (see e.g., 
\cite[Lemma~1.4]{nori:Be2})
that the derived category $D^b(C(X))$, with $R=k=\C$ and $\mathcal{N}$
=finite dimensional vector spaces, is equivalent to the triangulated 
category of ``cohomologically constructible sheaves'' 
as it is referred to in
the Riemann-Hilbert correspondence.  In view of Beilinson's
result (\cite[Main~theorem~1.3,~page~29]{nori:Be2}), one may re-state
the Riemann-Hilbert correspondence (\cite[Thm.~14.4]{nori:Bo}) as follows:

\medskip
 
\emph{The derived category of bounded complexes of regular holonomic modules 
is equivalent to the derived category of bounded complexes of
constructible sheaves of complex vector spaces.}

\begin{thm*} Assume that $R$ is a field, and that
 $\mathcal{N}$=all finite dimensional vector spaces. 
 Let $f:X \to Y$ be a morphism. Then the functors
 $\derived ^qf_*|C(X)$ are effaceable for all $q>0$. 
\end{thm*}

For the next theorem, assume that $R$ is commutative and recall that
for sheaves $F,G$ on $X$, we have the sheaf   (of $R$-modules) 
$\HHom(F,G)$ on $X$ and the 
derived functors $\derived^q\HHom(F,\cdot)=\EExt^q(F,\cdot)$.

\begin{thm*} With $R$ and $\mathcal{N}$ 
 as in Theorem~4 above, the functors \linebreak $\EExt^q(P,\cdot)|C(X)$
 are effaceable for all $q>0$ and for all constructible
sheaves $P$ on $X$.   
\end{thm*}

The minimal hypotheses on $R$ and $\mathcal{N}$  under
which Theorems~3(b),~4 and~5 are valid remain unclear. 
It is hoped that Theorems~1--5, or at least the method of proof,
will help in constructing a category of motivic sheaves. The
reader will observe that the proofs of Theorems~1 and 2
 require nothing more than the formal properties
 of the operations $f^*$ and 
$\derived^ qf_*$ while Theorems~3, 4 and 5 require the internal
 $\HHom$. That the analogues of Theorems~1--5 hold
for  $\mathbb{Q\,}_l$- sheaves (see \cite[page~84]{nori:SGAfourhalf}) is 
immediate from the proofs given here. 

The Noether normalisation lemma plays a major role in the proofs
of the sheaf-theoretic version of the basic lemma and Theorem~1. These proofs 
appear in the first two sections.
Some care has been taken to keep these sections self-contained
and elementary (modulo the use of the Leray spectral sequence and the 
proper base change theorem). For this purpose, direct proofs
of weak constructibility appear in these sections.  
In the third section, the 
remaining theorems are essentially deduced from Theorem~1, 
Whitney stratifications, devissage, and a 
reduction from \cite{nori:SGA4}, \cite{nori:SGAfourhalf} 
(embedding a sheaf $F$ on $X$ in 
$i_*i^*F \oplus j_*j^*F$ when the images of $i$ and $j$ 
cover $X$.

The first proof of Theorem~3(b), not the one
given here, was found with some aid in homological algebra from
P. Deligne. The version of Lemma~3.2(b) given here is based on a
remark of V. Srinivas.    The author thanks both P. Deligne and V. Srinivas
for useful discussions, and most of all, A. Beilinson for 
taking the trouble to explain in detail the very first  
consequences of his far-reaching Lemma~3.3.

\section{The basic lemma for sheaves}

A subfield $k\subset\C$ remains fixed throughout the paper. The
phrases ``sheaf on $X$'', ``weakly constructible sheaf'', etc. are
as in the introduction. All varieties and morphisms considered
are defined over $k$. All sheaves considered are sheaves 
of $R$-modules. 
\emph{The cohomology group $\homology ^q(X(\C),F)$
will always be denoted simply by $\homology ^q(X,F)$.}
The lemma below evidently implies the theorem of M. Artin
: ``$\homology^q(X,F)=0$ for $q> \dim(X)$, where $X$
 is affine '' (see \cite[Expose~XIV,~Corollary~3.2]{nori:SGA4}) and 
the analogous results of Hamm and Le (see \cite{nori:HL1} and \cite{nori:HL2})
 for weakly constructible sheaves.
It is unclear whether or not the  powerful relative version, 
due to M. Artin (\cite[Thm.~3.1,~Expose XIV]{nori:SGA4}), 
 can also be obtained directly. 

Beilinson's Lemma~3.3 of \cite{nori:Be2}
covers (at least) the case $F$ constructible, $R$ commutative
Noetherian, 
$\mathcal{N}$= all finitely generated modules, in the lemma 
below. Beilinson works in fields of all characteristics.  
Remark~1.1 below holds for his proof as well.
It would be nice to get a proof of 
this special case of Beilinson's very general lemma
 via Morse theory or the method of pencils. 
This would have the double advantage of covering the general case
and proving the theorems of M. Artin and Andreotti-Frankel
as well.

\begin{bslem}[A. Beilinson]
Let $F$ be a weakly constructible 
sheaf on an affine variety $X$. Let $n=\dim (X)$. 
Then there is a Zariski open $U\subset X$ with the properties below,  
where 
$j:U\to X$ denotes the inclusion, and
$F'=j_!j^*F\subset F$. 
\begin{enumerate}
\item[{\rm (1)}] $\dim (Y)<n$, where $Y$ is the complement of $U$ in $X$,
\item[{\rm (2)}] $\homology ^q (X,F')=0$ for $q\neq n$.
\item[{\rm (3)}]  There is a finite subset $E\subset U(\C)$ 
and an $R$-module isomorphism of 
$\homology ^n (X,F')$ with 
$\oplus \{ F_x:x\in E\}$.
\end{enumerate}
\end{bslem}
                
\begin{remark}
With $F$ on $X$ as in the lemma above, let $X'$ be
the largest Zariski open subset of $X$ so that $X'$ is smooth
and $F|X'$ is locally constant. The open $U$ in the statement
of the lemma \emph{depends only on} $X'$ \emph{and not on} $F$,
as can be seen from the proof.
 
The proof given uses the Noether normalisation lemma twice. If
in both these places, one uses general linear projections instead, one
sees that the $U$ in the lemma can be chosen
to contain any given finite subset of $X'(\C)$. In particular,
\emph{the $U$ for which the lemma holds cover $X$, if $F$ is locally
constant and if $X$ is smooth}. 
\end{remark}

\begin{remark}
The first form of the basic lemma, 
as given in the introduction,
is an immediate consequence. If $R_X$ denotes the constant sheaf
on $X$ and $v: V\to X$ is a Zariski open immersion with $W$
as its complement, then the sheaf cohomology
$\homology ^q(X,F)$, where $F=v_!v^*R_X$  
coincides with the singular cohomology
$\homology ^q(X,W;R)$. So, the above lemma applied to $F$ 
and $R=\Z$ yields
$Y$ with $\dim (Y)<n$ and $\homology ^q(X,Y\cup W)=0$
 for $q\neq n$, and   $\homology ^n(X,Y\cup W)=0$ 
is a free Abelian group. The universal
coefficient theorem gives the same result for all $R$ and
for homology as well.
\end{remark}

\begin{proof}[of the basic lemma (second form)]
Because the direct image of a locally
constant sheaf under a covering projection is locally constant,
it follows easily that the direct image of a weakly
constructible sheaf under a finite morphism is weakly 
constructible.   
With $X$ and $F$ as in the lemma, by Noether normalisation,
we have a finite morphism  $\pi:X\to \A^n$. Assuming the lemma
for the sheaf $\pi_*F$ on $\A^n$, we get a non-empty 
Zariski open $V\subset \A^n$ with the desired properties.

Putting $U=\pi^{-1}V$ and denoting the inclusions of $U$ 
in $X$ and $V$
in $\A^n$ by $j$ and $v$ respectively, we see that
$\pi_*j_!j^*F=v_!v^*\pi_*F$. Consequently 
 $\homology^q(X,j_!j^*F)=\homology^q(\A^n,v_!v^*\pi_*F )$,
and by the very choice of $V$, the latter vanishes for
 $q \neq n$. This proves part (2) of the lemma. 
Part (3) follows because $(\pi_*F)_y$
is the direct sum of $F_x$ taken over $x \in \pi^{-1}y$.

  It remains to prove the lemma for affine space.
We will proceed
by induction on dimension. Given a weakly constructible sheaf $F$
on $\A^n$, choose a \emph{nonconstant} $f$ in the co-ordinate
ring of $\A^n$ so that the restriction $F|\mathrm{D}(f)$ is
locally constant, where $\mathrm{D}(f)$ and $\V(f)$, are, as usual,
$\{ x|\,f(x)\neq 0\}$ and $\{x|\,f(x)=0\}$ in $\A^n$
respectively. 
Clearly, $F$ may be replaced by its subsheaf $d_!d^*F$ where
$d$ denotes the inclusion of $\mathrm{D}(f)$ in $\A^n$.
\emph{Thus we will assume that $F|\V(f)=0$.} 
Next,
after a linear change of variables, we may assume that $f$
is monic in the last variable. 
Denote by $\pi : \A^n \to
\A^{n-1}$ the projection $(x_1,x_2,\ldots ,x_n)\mapsto 
(x_1,x_2,\ldots, x_{n-1})$. The restriction $\pi |\V(f)$ is now
both \emph{finite} and \emph{surjective}. For $y\in \A^{n-1}(\C)$,
we consider the cohomology groups $\homology ^q
(\pi^{-1}y,F|\pi^{-1}y)$. That this vanishes for $q>1$  
is standard (e.g., Remark~1.4 below). 
Because $\V(f)\cap \pi^{-1}y$ is \emph{non-empty} and the restriction of
$F$ to this intersection is zero, we see that 
$\homology^0(\pi^{-1}y ,F|\pi^{-1}y)=0$ as well. 
Because $\pi |\V(f)$
is a finite morphism, we may apply VPBC (Proposition~1.3A below)
  to conclude that $\derived^q\pi_*F=0$
for all $q\neq 1$.  
For a non-empty Zariski open $U'\subset \A^{n-1}$ (that will be chosen
later),
let $U=\pi^{-1}U'$.
Let $j':U'\to \A^{n-1}$ and $j:U\to \A^n$ denote the given inclusions.
Put $G=\derived^1\pi_*F$. Let $G'=j'_!{j'}^*G$ and let $F'=j_!j^*F$.
From Corollary~1.3C below, it follows that 
$ \derived^q\pi_*F'=0$ for $q\neq 1$ and also that 
$G'\to \derived^1\pi_*F'$ is an isomorphism.
 From the  
Leray spectral sequence, we deduce that 
$\homology^q(\A^{n-1},G')$ is isomorphic to
$\homology^{q+1}(\A^n, F')$ for all $q$.
Checking easily that $G$ is weakly constructible 
(see Remark~1.5), 
apply the basic lemma to the pair  $(\A^{n-1}, G)$ to get a non-empty 
$U'\subset \A^{n-1}$ with the desired properties. It follows that 
$\homology^q(\A^n, F')$ vanishes for all $q\neq n$.
Finally, for part (3) of the lemma, 
$\homology^n(\A^n, F')=\homology^{n-1}(\A^{n-1},G')$
is isomorphic to a finite direct sum of stalks $G_y$ 
for $y\in U'(\C)$ by the induction hypothesis. But each such 
$G_y=\homology ^1(\pi^{-1}y,F|\pi^{-1}y)$ is
 isomorphic to a finite direct sum of stalks of $F|\pi^{-1}y$
from Remark~1.4 below.
This completes the proof of the basic lemma (modulo 1.3--1.5).
\qed\end{proof}

\emph{The statements 1.3A--C below are for topological
spaces and for sheaves of Abelian groups. In (H3) below, a 
topological space $Z$ is ``simply connected'' if every
locally constant sheaf $F$ on $Z$ is a constant sheaf. This is weaker
 than 
demanding that $Z$ is path-connected and has trivial fundamental group. 
If $f:L \to M$ is
 a fiber bundle(this is so in all applications), then (H2) holds locally on 
$M$}.

\begin{nprop}[Variation of proper base-change]
Let $\overline{f}\!:\!\overline{L}\! \to\nolinebreak\! M$
be a proper continuous map, where $\overline{L}$ and $M$
are second countable locally compact Hausdorff spaces. Let
$L \subset  \overline{L}$ be open and let $A$ be a sheaf on $L$.
Put $f=\overline{f}|L$. Assume (H1), (H2) and (H3) below.
\begin {enumerate}
\item[{\rm (H1)}] There is a closed subset $L_1\subset L$ so that
$f|L_1$ is proper and $A|L-L_1$ is a locally constant sheaf.
\item[{\rm (H2)}] For every point $x \in  \overline{L}-L$ 
there is a neighbourhood  $U$ of $x$ in $\overline{L}$ so that 
$\overline{f}(U)$ is open and 
$\overline{f}|(U,U \cap L) \to \overline{f}(U)$ is a fiber
bundle pair.
\item[{\rm (H3)}] Every point $m \in M$ has a fundamental system of 
simply connected neighbourhoods.  
\end{enumerate}
Then $\derived^qf_*A_m \to \homology^q(f^{-1}m, A|f^{-1}m)$
is an isomorphism for all $m \in M$.
\end{nprop}

\begin{ncor} 
With notation and assumptions as in Proposition~1.3A \linebreak above, let
 $g:M' \to M$ be
continuous with $M'$ second countable locally compact Hausdorff. 
Put $L'=M' \times
_M\,L$. Let $f':L' \to M'$ and $g':L' \to L$ denote the projections.
Then
  $g^*\derived ^q f_* A \to \derived^q f'_* {g'}^*A$ 
is an isomorphism.
\end{ncor}

\begin{nncor}
With notation and assumptions as in Proposition~1.3A above, the
homomorphism $B \otimes \derived ^qf_*A \to 
\derived ^q (A \otimes f^*B)$ is an isomorphism for all
sheaves $B$ on $M$ of the type $j_!j*(\Z _M)$ 
for an open immersion $j:V \to M$. 
\end{nncor}

\begin{proof}[of Proposition~1.3A] 
Let $L_m$ and $\overline{L}_m$ be the 
fibers over $m \in M$ in $L$ and $\overline{L}$ respectively.
 Let $j:L \to \overline{L}$ and $j_m:L_m \to \overline{L}_m$
denote the inclusions. We have the spectral sequences:
\begin{equation*}
\begin{split}
\mathrm{E}_2^{p,q}&=(\derived ^p\overline{f}_* \derived^q j_*A)_m \
\Rightarrow (\derived^{p+q}f_*A)_m \text{ and }\\
\mathrm{E}_2^{p,q}&=\homology^p (\overline{L}_m,
\derived ^q(j_m)_*A|L_m) \Rightarrow \homology^{p+q}    
(L_m,A|L_m)
\end{split}
\end{equation*}
and a homomorphism from the first to the second. The second is
the Leray spectral sequence for the sheaf $A|L_m$ and the
inclusion $j_m$. The first is obtained by taking stalks at $m$ of
the Leray spectral sequence for the composite $L \to \overline{L} \to\nolinebreak M$ 
and the sheaf $A$. By the proper base-change theorem, we see that
the $\mathrm{E}_2^{p,q}$ terms of the first spectral sequence
coincide with $\homology ^p( \overline{L}_m, \derived^q j_*A| \overline{L}_m)$.
To prove the proposition, it suffices to check that
 $\derived^qj_*A|\overline{L}_m \to 
\derived^q(j_m)_*(A|L_m)$
is an isomorphism. This arrow evidently
induces an isomorphism of stalks
at $x \in L_m$. For the rest, one observes
that (H1), (H2) and (H3) imply (H4) below and then notes that for 
$x \in \overline{L}_m -L_m$ the required 
isomorphism on stalks is a consequence. 
In any ``geometric situation''
 where $M$ is triangulable, this is immediate. That (H4)
and our assumptions on the topology of $L$ and $M$ are
adequate to draw the same conclusion is left to the reader.    
 This may be seen, for instance, by using the Kunneth formula. 

\medskip

\noindent (H4): Every $x \in \overline{L}-L$ has a neighbourhood $U\subset
\overline{L}$ with a continuous $h: (U, U\cap L) \to (V,W)$ so
that 
\begin{enumerate}
\item[{\rm (i)}] $\overline{f}(U)$ is open in $M$, 
\item[{\rm (ii)}] $(h,\overline{f}|U)$ induce a homeomorphism $(U, U\cap L) 
\to (V,W) \times \overline{f}(U)$, and 
\item[{\rm (iii)}] there is a sheaf $C$ on $W$ and an isomorphism 
$(h|U \cap L)^*C \to A|U \cap L$. 
\end{enumerate}
\qed\end{proof}
 
\begin{proof}[of Corollary~1.3B] Once it is noted that the hypothesis
(H4) in the proof of Proposition~1.3A is stable under base-change,
we see that stalks of both sheaves at $m'\! \in\! M'$ are compatibly 
isomorphic to $\homology^q(f^{-1}m, A|f^{-1}m)$ where $g(m')=m$
(this argument is standard, see e.g., 
the proof of the operation $\derived^qf_!$  being stable
under base change in \cite[page~49]{nori:SGAfourhalf}).  
\qed\end{proof}

\begin{proof}[of Corollary~1.3C] Denote by $S^qB$ the natural
 homomorphism $B \otimes \derived ^qf_*A \to 
\derived ^q (A \otimes f^*B)$.  
 Let $i:F \to M$ be the the inclusion
of the complement of the given open subset $V$. Now $S^q\Z_M$
is tautologically an isomorphism. Putting $g=i$ in Corollary~1.3B,
we see that $S^qi_*\Z_F$ is an
isomorphism. That $S^qj_!\Z_V$ is an isomorphism follows from
the long exact sequence of $S^q$ obtained from 
the short exact sequence: 
$$
\xymatrix@1{0 \ar[r] & i_*\Z_F \ar[r] & \Z_M \ar[r] & j_!\Z_V
\ar[r] & 0.
}
$$
\qed\end{proof}

\setcounter{thm}{3}
\begin{remark}
Consider a weakly constructible
sheaf $F$ on $\A ^1$ whose restriction to the complement of a non-empty
finite subset $S\subset \C$ is locally constant. Let $T$ be a tree embedded in 
$\C$ with $S$ as its set of vertices, and $E$ as its
set of edges. Let $b(e)$ denote the barycenter
 of an edge $e \in E$.  One checks easily that
\begin{enumerate}
\item[(a)] $\homology^q(\A ^1,F)\to \homology^q(T,F|T)$ is an isomorphism,
and 
\item[(b)] the open covering $\{ st(v): v \in S\}$  
 of $T$
give a Leray covering for $F|T$ (see e.g., 
\cite[Chapter~VIII,~Prop.~8.1.4.(ii),~page~323]{nori:KS}), 
for a more general statement). 
\end{enumerate}
The Cech complex for this covering is the arrow  
$ \oplus_{s \in S}F_s \to \oplus_{e \in E}F_{b(e)}$. Thus,  
 if $F|S=0$, then $ \oplus_{e \in E}F_{b(e)} \to\homology^1(T,F|T)$
 is an isomorphism. This gives part (3) of the basic lemma for the affine line.
\end{remark}

\begin{remark}
Let $\pi: X \times \A^1 \to X$ denote
the projection. Let $V\subset X \times A^1$
be a Zariski closed subset so that $\pi|V$ is a finite 
surjective morphism. Let $A$ be a sheaf on $X \times \A^1$ so 
that $A|V=0$ and the restriction of $A$ to the complement
$V$ is locally constant. We sketch briefly a proof that 
$\derived^1\pi_*A$ is weakly constructible.

If $\pi |V$ is a finite etale morphism, then $(X \times \A^1,V)
\to X$ is a fiber bundle pair. Consequently 
$\derived^1\pi_*A$ is a locally constant sheaf. 
For the general case,
express  
$X$ as a finite disjoint union of Zariski locally closed smooth
subvarieties $X_i$ so that $(\pi ^{-1}X_i \cap V)_{red} \to X_i$ 
is a finite etale morphism. By Corollary~1.3B, we see
that $\derived ^1 \pi _* A|X_i$ maps isomorphically 
to the locally constant sheaf 
$\derived ^1 (\pi |X_i \times \A^1 )_*A|X_i \times \A^1)$. 
This proves the weak constructibility of $\derived^1\pi_*A$.
\end{remark}

\begin{prop}Let $F$ be a 
constructible sheaf on a variety $X$. Then $\homology
^q(X,F)$ is an object of $\mathcal{N}$ for all $q \geq 0$.
\end{prop}

\begin{proof} Part (3) of the second form of the basic
lemma, combined with induction on dimension, proves the 
proposition when $X$ is affine. The general case now
follows  by induction on the number $n$ of affine open sets
 required to cover $X$, and the Mayer-Vietoris sequence (the
 same technique is used in the proof of Proposition~3.6).
\qed\end{proof}

\section{The cohomology of affine space} 

This section begins with the proof of Theorem~1
as stated in the introduction. The conventions and notation
are as in the previous section. The proof is
by induction on $n$, the inductive step relying crucially 
 on a  canonical proof for the case 
$n=1$ that works with parameters (Proposition~2.2 below). We
then deduce the effaceability of cohomology for affine varieties
(Corollary~2.4).

We recall the statement of the theorem.

\renewcommand{\thethm}{1}

\begin{thm} Every constructible  $F$ on affine space
is a subsheaf of a constructible $G$ so that
$\homology^q(\A^n,G)=0$ for all $q>0 $.
\end{thm}

\begin{proof} We begin as in the proof of the basic lemma.
Given a constructible sheaf $F$
on $\A^n$, choose a nonconstant $f$ in the co-ordinate
ring of $\A^n$ so that the restriction $F|\mathrm{D}(f)$ is
locally constant, where $\mathrm{D}(f)$ and $\V(f)$, are, as before,
$\{ x|\,f(x)\neq 0\}$ and $\{x|\,f(x)=0\}$ in $\A^n$. 
The projection $\pi:\A^n\to \A^{n-1}$ is, once again, arranged
so that $\pi|\V(f)$ is a finite (and surjective) morphism. 
Denote by 
$j:\mathrm{D}(f)\to \A^n$ and $i:\V(f)\to \A^n$ the inclusions.

From Proposition~2.2 below (with $X=\A^{n-1}$ and $A=j_!j^*F$ )
we see that $j_!j^*F$ is a subsheaf of 
a constructible $H'$
on $\A^n$ so that $\derived^q\pi_*H'=0$ for all $q\geq 0$.
Taking push-outs, we obtain the commutative diagram
of exact sequences below:
$$
\xymatrix{       
0 \ar[r]\ar[d] &  j_!j^*F \ar[r]\ar[d]^-{\text{mono}}  &     
F \ar[r]\ar[d]^-{\text{mono}}  & i_*i^*F \ar[r]\ar[d]^-{\text{iso}} & 0 \\
0 \ar[r] &  H'  \ar[r] & H    \ar[r]  &H''  \ar[r] & 0. \\
}   
$$
We will say that a homorphism $P\to Q$ of sheaves on $\A^n$
is a \emph{$\pi$-isomorphism} if  
$\derived^q\pi_*P \to \derived^q\pi_*Q$
is an isomorphism for all $q\geq 0$. We shall see that all arrows
in the commutative diagram below are $\pi$-isomorphisms.
$$
\xymatrix{
\pi^*\pi_*H \ar[r]\ar[d]  & H\ar[d] \\
\pi^*\pi_*H'' \ar[r] &  H''.     \\
}
$$
We have the implications:
\begin{equation*}
\begin{split}
\derived^q\pi_*H'& =0 \mbox{ for all } q\geq 0\\
&\Rightarrow H\to H'' \mbox{ is a } \pi\mbox{-isomorphism}\\
&\Rightarrow \pi_*H \to \pi_*H'' \mbox{ is an isomorphism}\\
&\Rightarrow \pi^*\pi_*H \to \pi^*\pi_*H'' \mbox{ is an isomorphism.}
\end{split}
\end{equation*}
Next, because $\pi |\mathrm{V}(f)$ is a finite morphism, we
see that $H''$ is constructible and that 
$\derived ^q \pi_*H''$ vanishes for $q>0$. For any sheaf $K$ 
on $\A ^{n-1}$, $\derived ^q \pi_*\pi^*K$ vanishes for $q>0$.
Putting $K=\nolinebreak \pi_*H''$, 
we deduce that $\pi^*\pi_*H'' \to H''$ is a $\pi$-isomorphism.
We conclude that the only remaining
arrow in the square $\pi^*\pi_*H \to H$ is 
itself a $\pi$-isomorphism. 

Now let  
$ \pi_*H\subset J$ with $J$ a sheaf on $\A^{n-1}$.
Taking push-outs once again yields the commutative diagram
of exact sequences below:
$$
\xymatrix{
0 \ar[r]\ar[d] & \pi^*\pi_*H \ar[r]\ar[d] & \pi^*J \ar[r]\ar[d] &
M \ar[r]\ar[d]^-{\text{id.}} & 0 \\
0  \ar[r] &  H \ar[r] & G \ar[r] & M \ar[r] & 0. \\
}
$$
The vertical arrows on the sides being $\pi$-isomorphisms, 
it follows
that $\pi^*J \to G$ is a $\pi$-isomorphism as well. It 
follows that 
$$\xymatrix@1{
\homology^q(\A^{n-1},J) \ar[r] & \homology^q(\A^n,\pi^*J) 
\ar[r] &  \homology^q(\A^n, G)}
$$ 
are isomorphisms, for
all $q\geq 0$. By the induction hypothesis, we may choose $J$  
satisfying $\homology^q(\A^{n-1},J)=0$
for all $q>0$ with $J$ constructible.  We now have  
 $F \subset H \subset G$  with $G$ constructible  
and $\homology^q(\A^n, G)=0$ for all $q>0$. 
The proof of the theorem is now complete 
modulo Proposition~2.2.
\qed\end{proof}

\renewcommand{\thethm}{2.1}
\begin{notation}
 We fix a variety $X$ and work in the category of $X$-schemes.
The product $\A^r \times X$ is denoted by $\A^r_X$.  
Let $p_i:\A^2_X \to \A^1_X\,$, $\pi:\A^1_X \to X\,$ and 
$\Pi:\A^2_X \to X$ denote the given projections, and let
 $\Delta :\A^1_X\to \A^2_X$ denote
the diagonal morphism. 
\end{notation}

For any sheaf $A$ on $\A^1_X$, we define $B$ on $ \A^2_X  $
by the exact sequence:  
$$\xymatrix@1{
  0\ar[r] & B \ar[r] & p_1^*A \ar[r] & \Delta_*A \ar[r] & 0.
}
$$
Applying $\derived ^q(p_2)_*$ 
to the above exact sequence gives the arrow: 
$$
\delta : A= (p_2)_* \Delta_*A   \to C=\derived ^1(p_2)_*B.
$$

\renewcommand{\thethm}{2.2}

\begin{prop}Let $V\subset \A^1_X$
be a Zariski closed subset so that $\pi|V$ is a finite 
surjective morphism. Let $A$ be a sheaf on $\A^1_X$ so 
that $A|V=0$ and the restriction of $A$ to the complement
$\A^1_X-V $ is locally constant. With  
$\delta : A \to C$ as above, we have:
\begin{enumerate}
 \item[{\rm (1)}] the above $\delta: A \to C$ is a monomorphism,
\item [{\rm (2)}] $\derived ^q\pi_*C=0$ for all $q\geq 0$,
\item [{\rm (3)}] $C$ is weakly constructible. If $A$ is 
constructible, so is $C$. 
\end{enumerate}
\end{prop}

\begin{proof} Applying
$\derived ^q(p_1)*$ to the short exact sequence that defines
$B$, we see that $\derived ^q(p_1)_*B=0$ for all $q$. From the
Leray spectral sequence applied to $\Pi=\pi\circ p_1$ we see that
$\derived ^q\Pi_*B=0$ for all $q$.

Put $Y=\Delta (\A ^1_X) \cup p_1^{-1}V$. Then $B|Y=0$ and
$B$ is locally constant on the complement of $Y$. Now 
 $p_2|Y$ is a finite surjective morphism. So, substituting 
$(\A^2_X,\A^1_X,p_2,B)$ for $(L,M,f,A)$ in Proposition~1.3A,
we get $\derived ^q(p_2)_*B=0$
for $q \neq 1$. By the Leray spectral sequence for
$\Pi=\pi\circ p_2$ we deduce that   
$\derived ^q\pi_*C=\derived ^{q+1}\Pi _*B=0$. This proves part (2)
of the proposition. That $C$ is weakly constructible 
follows from Remark~1.5. By Proposition~1.3A and Remark~1.4,
 any stalk of $C$, being isomorphic to
 a finite direct sum of stalks of $B$,
 is indeed an object of $\mathcal{N}$ if $A$ is constructible.
This proves part (3). Finally, 
$ \pi^*\derived ^q\pi_*A \to \derived^q(p_2)_*p_1^*A$ is an
isomorphism by Corollary~1.3B, because $V \to X$ is a 
finite morphism. Furthermore, $V \to X$ being surjective, 
the stalks of these sheaves vanish for $q=0$, and this proves
part (1) of the proposition. 
\qed\end{proof} 

\renewcommand{\thethm}{2.3}

\begin{remark}
From the above proof, we see that
$A \mapsto C$ is, in fact, an exact functor for $A$ as
in the proposition. Furthermore $\derived^qp_2*$ gives the
short exact sequence:
$$\xymatrix@1{
0 \ar[r] & A \ar[r] & C \ar[r] & \pi^*\derived ^1\pi_*A \ar[r] & 0.
}
$$
\end{remark}

\renewcommand{\thethm}{2.4}

\begin{cor} Every constructible
 sheaf $F$ on an affine variety $X$
is a subsheaf of a constructible 
sheaf $G$ on $X$ so that the induced homomorphism
$\homology^q(X,F)\to \homology^q(X,G)$ is zero for all $q>0$.
\end{cor}

\begin{proof} Let $i:X \to \A^n$ be a closed immersion.
From Theorem~1, we have 
$i_*F\subset T$ with $T$ constructible so that $\homology^q(\A^n,T)
=0$ for all $q>0$. The commutative diagram:
$$
\xymatrix{  
\homology^q(\A^n,i_*F) \ar[r]\ar[d] & \homology^q(\A^n,T)\ar[d]\\  
\homology^q(X,i^*i_*F) \ar[r] & \homology^q(X,i^*T).
} 
$$
shows that $G=i^*T$ is the desired sheaf.
\qed\end{proof}

\section{Ext and higher direct images}

\paragraph{\textbf{Definition and notation 3.1:}}Admissibility, 
$F[U]$ and $R_X$.

An object $F$ of $Sh(X)$ is admissible if the functor $\Ext^q_{Sh(X)}(F,\cdot))
|C(X)$ is effaceable for every $q>0$.

Let $F$ be a sheaf on $X$ and let $U$ be Zariski open in $X$.
If $j:U \to X$ denotes the inclusion, then $F[U]=j_!j^*F$. 
 
$R_X$ always denotes the constant sheaf on $X$ with all stalks 
equal to $R$.

\renewcommand{\thethm}{3.2}

\begin{lem} Let 
$$
\xymatrix@1{
0\ar[r] & F'\ar[r] & F \ar[r] & F'' \ar[r] & 0
}
$$ 
be an exact sequence in $Sh(X)$. 
\begin{enumerate}
\item[{\rm (a)}] Assume $F''$ is admissible
and at least one of  {$F', F$} is admissible. Then all three are
admissible. 
\item[{\rm (b)}] Assume $F$ and $F'$ are admissible, and also 
that  
$$
\coker ( \Hom_{Sh(X)}(F,\cdot) \to \Hom_{Sh(X)}(F',\cdot))|C(X) 
$$ 
is effaceable. Then  $F''$ is admissible.
\end{enumerate}
\end{lem}

\begin{proof} Consider the long exact sequence of Ext:
$$
\xymatrix@-1pc{
\Ext^{q-1}_{Sh(X)}(F',\cdot) \ar[r] &
\Ext^q_{Sh(X)}(F'',\cdot) \ar[r] &  \Ext^q_{Sh(X)}(F,\cdot)    
\ar[r] & \Ext^q_{Sh(X)}(F',\cdot). 
}
$$
Let $P$ be a constructible sheaf on $X$ and let $q>0$. For
part (b), there is a monomorphism $P \to Q$ with $Q$ constructible
that effaces $\Ext^q_{Sh(X)}(F,P)$. Next choose a monomorphism
$Q \to S$ with $S$ constructible 
that effaces $\Ext^{q-1}_{Sh(X)}(F',Q)$ if $q>1$. If $q=1$,
then $Q \to S$ is chosen so as to efface the
cokernel of $ \Hom_{Sh(X)}(F,Q)    
\to \Hom_{Sh(X)}(F',Q) $. The long exact sequence
of $\Ext$ shows that $P \to S$ effaces $\Ext^q_{Sh(X)}(F'',P)$.
 Part (a) follows in the same manner. 
\qed\end{proof}

Corollary~3.3 below is an immediate consequence of Lemma~3.2(a).

\renewcommand{\thethm}{3.3}
\begin{cor} If $0=F^n\subset F^{n-1} \subset \cdots \subset
F^1 \subset F^0=F$ and if all the $F^i/F^{i+1}$ are admissible,
then $F$ is admissible.
\end{cor}

\renewcommand{\thethm}{3.4}

\begin{prop} If $F$ is constructible
 on $X$ and $f:X \to Y$ is a morphism, then $\derived^qf_*F$ 
is constructible for all $q \geq 0$.
\end{prop}

\begin{proof} This proposition is by now well known; the 
conventional assumption ``$\mathcal{N}$= 
all Noetherian $R$-modules'' plays no role in the proof.
It suffices to prove the assertion 
in the two cases:(i) $f$ is proper, and (ii) $f$ is an open
immersion. 

When $f$ is proper, the proper base-change theorem, the
 existence of Whitney stratifications, and Thom's
isotopy lemma (see \cite[page~41]{nori:GM}) together show that 
 $\derived^qf_*F$ is weakly constructible. By proper base-change,
the stalk of this sheaf at $y \in Y(\C)$ coincides with
$\homology^q(f^{-1}y, F|f^{-1}y)$, which is an object of 
$\mathcal{N}$, by Proposition~1.6. Thus 
  $\derived^qf_*F$ is constructible.

When $f$ is an open immersion, one chooses a Whitney  
stratification: $Y$= finite disjoint union of $Y_i$ so
that $f_!F|Y_i$ is locally constant for all $i$. It follows
that
$\derived^qf_*F|Y_i$ is locally constant, because we 
are in a ``product situation''. The links $L(p)$ (see \cite[page~41]{nori:GM})
being finite simplicial complexes, we're assured (e.g., by 
\cite[Prop.~8.1.4(ii)]{nori:KS}) that the stalks of  $\derived^qf_*F|Y_i$ 
are objects of $\mathcal{N}$. This completes the proof.
\qed\end{proof}

\renewcommand{\thethm}{3.5}
\begin{lem}Let $F$ be an object of $Sh(X)$. 
\begin{enumerate}
\item[{\rm (a)}] If $F$ and $F[U]$ are
both admissible where $U \subset X$ is Zariski
 open, then so is $F/F[U]$. 
\item[{\rm (b)}] Let $V$ and $W$ be Zariski
open in $X$. Assume that $F[V]$, $F[W]$ and $F[V \cap W]$
are admissible. Then $F[V \cup W]$ is also admissible (see 3.1 for notation).
\end{enumerate}
\end{lem}

\begin{proof} Both parts follow from Lemma~3.2(b). 
For (a), it is sufficient to show that 
every constructible $P$ is
contained in a constructible $Q$ with  
$ \Hom_{Sh(X)}(F,Q) \to \Hom_{Sh(X)}(F[U],Q)$ surjective.
 Denoting by $j$ and $i$ the inclusions of $U$ and its complement
 $Y$ in $X$ respectively, 
the natural arrow  
$P \to Q=j_*j^*P \oplus i_*i^*P$ is a monomorphism that has the  
desired property because  $ \Hom_{Sh(X)}(F[U], i_*i^*P)=0$ and 
$$
\Hom_{Sh(X)}(F[U], j_*j^*P )=  \Hom_{Sh(X)}(F, j_*j^*P)=
  \Hom_{Sh(U)}(j^*F, j^*P).
$$
That $Q$ is constructible has been shown in Proposition~3.4 above.

For (b), consider the Mayer-Vietoris exact sequence:
$$\xymatrix{
0 \ar[r] & F[V \cap W] \ar[r] & F[V] \oplus F[W] \ar[r] & F[V \cup W] \ar[r] & 
0.
}
$$
We obtain $Q$ from $P$ as in the proof of part (a) by
putting $U=V \cap W$. The surjectivity of  
$\Hom_{Sh(X)}(F,Q) \to \Hom_{Sh(X)}(F[V \cap W],Q)$ 
certainly implies the surjectivity of 
$$ 
\Hom_{Sh(X)}(F[V],Q) \oplus \Hom_{Sh(X)}(F[W],Q) 
\to \Hom_{Sh(X)}(F[V \cap W],Q).
$$ 
This completes the proof of the lemma.
\qed\end{proof}

\renewcommand{\thethm}{3.6}

\begin{prop}Let $U \subset X$ be Zariski open. Then $R_X[U]$ is admissible 
(see 3.1 for notation).
\end{prop}

\begin{proof} We proceed  by induction on the number $n$ of
affine open subsets required to cover $U$. Now 
assume that $U$ is affine (the case: $n=1$). 
Corollary~2.4 says that $R_U$ is admissible in $Sh(U)$. It 
follows that  $R_X[U]$ is admissible in $Sh(X)$ 
(see e.g., Remark~3.8(a)).  
For the general case, write $U=V \cup W$ where $V$ is covered by
$n-1$ affine opens and $W$ is affine open. Because $X$ is 
separated, by Chevalley, $V \cap W$ is covered by $n-1$ affine
open subsets in $X$, so we may assume the result for $V,W$ and
 $V \cap W$. The admissibility of $R_X[U]$ follows from Lemma~3.5(b).  
\qed\end{proof}

\begin{proof}[of Theorem~2]
 The proof of the admissibility of $R_X$ in Proposition~3.6 above actually
proves Theorem~2. In any case, the admissibility plus
the vanishing of $\homology^q(X,F)$ for $q>2 \dim.(X)$
implies the theorem.
\qed\end{proof}

\renewcommand{\thethm}{3.7}
\begin{remark}
Let $F$ be a constructible sheaf on $X$ with the additional property
that $X$ is the union of a finite collection of Zariski locally closed 
subsets $Y_i$ so that for each $i$, $F|Y_i$ is a \emph{constant 
sheaf}. Assume that $R$ is Noetherian and that $\mathcal{N}$ is
the category of all finitely generated $R$-modules. In this case,
one can check that $F$ is admissible. We will not use this however
in the sequel. 
\end{remark} 

\renewcommand{\thethm}{3.8}

\begin{remark}[Avoiding the use of injective objects] 
Let $F: \mathcal{A} \to\nolinebreak \mathcal{B}$
and $G:\mathcal{B} \to \mathcal{A}$ be adjoint functors,
viz., 
$$
\phi (A,B) :\Hom(FA,B) \to \Hom(A,GB)
$$ 
for all objects $A$ of $\mathcal{A}$
 and $B$ of $\mathcal{B}$, is an 
\emph{isomorphism} functorial in $A$ and $B$ (see \cite[Chapter~IV,~page~80]{nori:Mac}). Assume that $\mathcal{A}$
and $\mathcal{B}$ are Abelian categories, and that $F$ is left
exact. Then it is standard (see \cite{nori:Toh}) that $G$ takes injectives to
injectives. We will not assume that $\mathcal{B}$ 
possesses injective objects. We will observe instead that the 
left exactness of $F$ ensures:
   
\emph{For every monomorphism} $u:GB \to A$, \emph{there is a monomorphism} 
\linebreak $v:B\to B'$ \emph{so that} $Gv:GB \to GB'$ \emph{factors through
 u, i.e. there is} $w:A \to GB'$ \emph{so that} $Gv=w \circ u$.
 \emph{In particular, if} $H: \mathcal{A} \to \mathcal{C}$ \emph
{is effaceable, then} $H \circ G$ \emph{is also effaceable.}
The  $v:B \to B'$ is obtained simply as the push-out of the
 natural arrow $\epsilon (B): FGB \to B$ with the monomorphism $Fu:FGB \to FA$.
Denote by $t:FA \to B'$ the resulting arrow. The $w$ is
then obtained as $Gt\circ \eta (A)$ where $\eta (A):A \to GFA$
is the natural arrow. For the dfns of $\epsilon (B)$
 and $\eta (A)$ see \cite[Thm.~1,~page~82]{nori:Mac}.
 The $(F,G)$ we are concerned with are:
\begin {enumerate}
\item [(a)] ($j_!$ ,$j^*$) where $j:U \to X$ is an open immersion and
 $\mathcal{A}=C(U)$ and $\mathcal{B}=C(X)$. In particular,
 the admissibility
of $Q$ on $U$ implies the admissibility of $j_!Q$ on $X$.

\item [(b)] ($f^*$ ,$f_*$) where $f:X \to Y$ is a morphism and 
 $\mathcal{A}=C(Y)$ and $\mathcal{B}=\nolinebreak C(X)$.

\item [(c)] ($P \otimes \cdot$ , $\HHom(P,\cdot)$). 
Here R is assumed to be commutative, and $P$ is weakly constructible 
on $X$ with all stalks finitely generated projective modules  
and $\mathcal{A}=\mathcal{B}=C(X)$. 
\end{enumerate}
\end{remark}

\renewcommand{\thethm}{3.9}
\begin{defn}[elementary, projectively elementary] 
A sheaf $F$ on $X$ 
is elementary if there are Zariski open  $V \subset U \subset X$
and a locally constant sheaf $L$ on $U$ so that $F$ is 
isomorphic to $j_!(L/L[V])$, where 
 $j:U \to X$ denotes the inclusion.

If all the stalks of $L$ are finitely generated projective 
 $R$-modules, then $F$ will be called \emph{projectively elementary}.  
\end{defn}

\renewcommand{\thethm}{3.10}
\begin{prop} Every weakly constructible
sheaf $F$ on $X$ has a finite filtration 
$$
0=F^n\subset F^{n-1} \subset\cdots \subset F^1 \subset F^0=F
$$
so that $F^r/F^{r+1}$ is a \emph{direct summand} of
 an elementary sheaf $T^r$ on $X$ for all $r$. Every stalk of
 $T^r$ is a finite direct sum of stalks of $F$. In particular,
 if all the stalks of $F$ are finitely generated projective 
$R$-modules, then the $T^r$ above are projectively elementary.
\end{prop} 

We will assume the above proposition and proceed. 
Proposition~3.10 is proved at the end of this section.

\begin{proof}[of Theorem~3]
We first assume that $F$ is projectively elementary.

Let $L,U,V,F,j$ be as in the dfn of ``projectively elementary''.
Because $L$ is a locally constant sheaf
on a variety $U$ with all stalks finitely generated projective,
we see that $\EExt^q(L,D)=0$ for all $q>0$ and for all $D$ on $U$.
The spectral sequence (see \cite[Thm.~4.2.1,~page~188]{nori:Toh}) shows that  
$\Ext^q_{Sh(U)}(L,D)=\homology^q(U,\HHom(L,D))$. 
From Theorem~2 and Remark~3.8(c), it follows that $L$ is admissible on $U$. 
Replacing $(L,U)$ by $(L|V,V)$ we see that that $L|V$ is admissible on $V$ as 
well. From Remark~3.8(a), we see that $L[V]$ is admissible on $U$. 
By Lemma~3.5(a),  $L/L[V]$ is admissible on $U$. By Remark~3.8(a) once again, 
 $F=j_!(L/L[V])$ is also admissible on $X$. 

Evidently, any direct summand of an admissible is admissible. The
admissibility of $P$ as in Theorem~3(a) now follows from Proposition~3.10 and 
Corollary~3.3. 
\qed\end{proof}

\renewcommand{\thethm}{3.11}

\begin{thm}Assume that 
$\mathcal{A'},\mathcal{A},
\mathcal{B'}, \mathcal{B}, F$ and $G$  
satisfy the properties (1),(2),(3) and (4) below.
It then follows that the functors 
$\derived^eG|\mathcal{B'}$ are effaceable 
for all $e>0$. 
\begin {enumerate}
\item[{\rm (1)}] $\mathcal{A'}$ is a full Abelian subcategory 
of an Abelian category $\mathcal{A}$. Every object of 
 $\mathcal{A}$ that is isomorphic to an object of 
$\mathcal{A'}$ is an object of $\mathcal{A'}$. The
category $\mathcal{A}$ possesses enough injectives.
For every object $A'$ of $\mathcal{A'}$ and for every
$q>0$, the functor $\Ext^q_{\mathcal{A}}(A', \cdot)|\mathcal{A'}$ 
is effaceable. 
\item [{\rm (2)}] $\mathcal{B'}$ is a full Abelian subcategory 
of an Abelian category $\mathcal{B}$. All the properties in
(1) above hold for $\mathcal{B'}$ and $\mathcal{B}$.
\item [{\rm (3)}] $G: \mathcal{B}\to \mathcal{A} $ is a left
exact functor so that $\derived^qGB'$ is an object of 
$\mathcal{A'}$ for all objects $B'$ of $\mathcal{B'}$
and for all $q \geq 0$. 
\item[{\rm (4)}] $F: \mathcal{A}\to \mathcal{B} $ is a left 
adjoint of $G$. In addition, $F$ is left exact. Furthermore, 
 $FA'$ is an object of $\mathcal{B'} $ for all objects $A'$ of $\mathcal{A'}$. 
\end{enumerate}
\end{thm}

\begin{proof}
Proceeding by induction on $e$, we
assume the effaceability of $\derived^qG|\mathcal{B'}$ 
for all $q$ such that $0<q<e$. Let $f:A' \to \derived^eGB'$ 
be a homomorphism in $\mathcal{A}$,  
where $A'$ and $B'$ are objects of  
$\mathcal{A'}$ and $\mathcal{B'}$ respectively. We will
find a monomorphism $B' \to B'_{e+1}$ in $\mathcal{B'}$ 
so that the composite   
$$
\xymatrix{A' \ar[r] & \derived^eGB' \ar[r] & \derived^eGB'_{e+1}
}
$$ 
is zero.
Consider the functor $\Hom(A', \cdot): \mathcal{A} \to Ab$ 
where $Ab$ is the category of Abelian groups. By (4) above,
note
 that $\Hom(A',\cdot)\circ G=\Hom(FA',\cdot)$ and that $G$
takes injectives to injectives.  From \cite{nori:Toh}, for every
object $B$ of  $\mathcal{B}$, 
we get the spectral sequence below, functorial in $B$:
$$\mathrm{E}^{p,q}_2(B)=\Ext^p_{\mathcal{A}}(A', \derived ^qGB) 
\Rightarrow \Ext^{p+q}_{\mathcal{B}}(FA',B).
$$
We put $f=f_0$ and $B'=B'_0$. 
The given $f_0$ belongs to $\mathrm{E}^{0,e}_2(B'_0)$. 
Its differential $\delta f_0$ belongs to $ \mathrm{E}^{2,{e-1}}_2(B'_0)$. 
By the induction hypothesis, there is a 
a monomorphism $B'_0 \to B'_1$ in $\mathcal{B'}$ 
 that effaces $\derived^{e-1}GB'_0$. Thus $B'_0 \to B'_1$     
induces zero on $\mathrm{E}^{2,e-1}_2$. It follows that 
$\delta f_0 \mapsto 0$ under $B'_0 \to B'_1$. Denoting the
image of $f_0$ in $\mathrm{E}^{0,e}_2(B'_1)$ by $f_1$ we
see that $f_1 \in \mathrm{E}^{0,e}_3(B'_1)       \subset
 \mathrm{E}^{0,e}_2(B'_1)$.
Continuing in this manner, using the induction hypothesis alone, 
we obtain a chain of elements 
$f_i \in \mathrm{E}^{0,e}_{i+2}(B'_i)$ for $0 \leq i <e$ 
and monomorphisms $B'_i \to B'_{i+1}$ 
in $\mathcal{B'}$ 
for $0 \leq i < e-1$  that take  
 $f_i$ to $f_{i+1}$. Note that  this is meaningful because    
$\mathrm{E}^{0,e}_{i+3} \subset \mathrm{E}^{0,e}_{i+2}$. 
 Next note that 
$$
\delta f_{e-1} \in \mathrm{E}^{e+1,0}_{e+1}(B'_{e-1})
=\Ext^{e+1}_{\mathcal{A}}(A', GB'_{e-1}).
$$ 
Assumptions (1) and (4), combined
with Remark~3.8, imply that the functor 
$\Ext_{\mathcal{A}}^{e+1}(A', \cdot) \circ G|\mathcal{B'}$
 is effaceable. Thus we get a monomorphism  
  $B'_{e-1} \to B'_e$ in $\mathcal{B'}$
that effaces  $\Ext^{e+1}_{\mathcal{A}}(A', B'_{e-1})$, and
as before, this monomorphism takes  $f_{e-1}$ to $f_e$ with 
$f_e \in  \mathrm{E}^{0,e}_{e+2}(B'_e)= 
\mathrm{E}^{0,e}_{\infty}(B'_e)$.
Finally, because $e>0$, by (2), there
is a monomorphism $B'_e \to B'_{e+1}$ in $\mathcal{B'}$  
that effaces
 $\Ext^e_{\mathcal{B}}(FA',B'_e)$. Because  
 $  \mathrm{E}^{0,e}_{\infty}(B'_e)$ is a quotient of
 $\Ext^e(FA', B'_e)$, it follows that the induced homomorphism
$\mathrm{E}^{0,e}_{\infty}(B'_e) \to  
\mathrm{E}^{0,e}_{\infty}(B'_{e+1})$ is zero. Consequently, 
$f_e \mapsto 0$ in $\Hom(A', \derived^eGB'_e) \to
 \Hom(A', \derived^e GB'_{e+1})$. It follows that 
 $f \mapsto 0$ in $\Hom(A', \derived^eGB') \to
 \Hom(A', \derived^e GB'_{e+1})$. This 
proves our claim. 
 
Now let $f$ be the identity
homomorphism of  $ \derived^eGB'$. The theorem follows.
\qed\end{proof}

\begin{proof}[of Theorem~4] Given a morphism $f:X \to Y$,
we substitute 
$$
(f^*,f_*, Sh(Y), C(Y),Sh(X),C(X))= 
(F,G, \mathcal{A},  \mathcal{A'}, \mathcal{B}, \mathcal{B'})
$$ 
in Theorem~3.11. The assumptions (1) and (2) of Theorem~3.11 are  
valid by Theorem~3(b); assumption (3) is valid by Proposition~3.4, and 
assumption (4) is evident. So Theorem~4 is a consequence of Theorem~3.11. 
\qed\end{proof}

\begin{proof}[of Theorem~5] Once again, we apply
Theorem~3.11 to  
$$
\mathcal{A}=\mathcal{B}=Sh(X),\   \mathcal{A'}=\mathcal{B'}=C(X), 
\mbox{ and } (F,G)= (P \otimes \cdot\;, \HHom(P,\cdot)),
$$ 
for any constructible sheaf $P$ on $X$. To check that the assumptions
of Theorem~3.11 are valid, it remains to show that $\EExt^q(P,Q)$
 is constructible if both $P$ and $Q$ are constructible. By 
Proposition~3.10, it suffices to check this for $P=j_!L$ where $L$
 is locally constant on $U$ and $j:U \to X$ denotes the 
 inclusion.  But $\EExt^q(j_!L,Q)=\derived^q j_*\HHom (L,j^*Q)$
which is constructible by Proposition~3.4. This completes the
proof of Theorem~5.  
\qed\end{proof}

\begin{proof}[of Proposition~3.10]
Let $F$ be a weakly constructible sheaf 
on a quasi-projective variety $X$. Let dim.supp.$F=d$.
We will first find  Zariski open 
$V \subset U\subset X$, a locally constant sheaf $L$
 on $U$, and an isomorphism of $F|U$ with a direct summand of 
$L/L[V]$  (see 3.1 and 3.9 for notation).  

Let $Z_1$ be the support of
$F$. Let $Z_2$ be the union of all the $d$-dimensional
irreducible components of $Z_1$.
 Let $Z_3$  
be the largest Zariski open subset of $Z_2$
 so that $F|Z_3$ is locally constant. Let $U' \subset X$
be an affine open subset so that 
$U'\cap Z_1 = U' \cap Z_3$ and $U' \cap Z_3$ is Zariski
dense in $Z_3$. Choose a 
finite surjective morphism $ U'\cap Z_1 \to \A^d$ and extend
 it to a morphism $h':U' \to \A^d$. Let $W \subset \A^d$
 be non-empty Zariski open. Put $U={h'}^{-1}W$, $Z=U \cap Z_1$ 
and $V=U-Z$. Denote the restriction of $h'$ to $U$ 
 by $h:U \to W$. Put $f=h|Z:Z \to W$. 
 The sheaf $G=F|Z$ is locally constant by assumption.
The $W$ is chosen so that 
$f:Z \to W$ is an etale morphism (note that $f$ 
is already assumed to be a finite
 morphism); this ensures that $f_*G$
is locally constant. It follows that $L=h^*f_*G$ is
also a locally constant sheaf. Note that $f^*f_*G=L|Z$ has
 $G$ as a direct summand because $f$ is etale.
Denoting by $i:Z \to U$ the inclusion, we see that 
 $i_*G=F|U$ is a direct summand of $i_*(L|Z)=L/L[V]$. 

From the above, we see that $F[U]$ is a direct
summand of an elementary sheaf. Put $QF=F/F[U]$. Because  
dim.supp.$QF<d$, induction on dim.supp.$F$ proves
the proposition with the length of the desired 
filtration bounded
above by 1+dim.supp.$F$, when $X$ is quasi-projective.
In the general case,  one proceeds  by Noetherian induction, 
replacing the hypothesis `` $U' \cap Z_3$ is Zariski
dense in $Z_3$'' by `` $U' \cap Z_3$ is non-empty''.
\qed\end{proof}

\myaddress{The University of Chicago, 5734 S.University  Ave, Il 60637 USA}
\myemail{nori@math.uchicago.edu}


\begin{thebibliography}{[XXXX]}

\bibitem[AF]{nori:AF} A. Andreotti, and T. Frankel, \emph{The Lefschetz 
theorem on hyperplane sections}, Ann. of Math. {\bf 69} No.2 (1959), 713--717. 

\bibitem[Be1]{nori:Be1} A.A. Beilinson, \emph{Notes on absolute Hodge 
cohomology}, Applications of algebraic $K$-theory to algebraic geometry and 
number theory, Part I, II (Boulder, Colo., 1983), 35--68, Contemp. Math. 
{\bf 55} Amer. Math. Soc., 1986. 

\bibitem[Be2]{nori:Be2} A.A. Beilinson, \emph{On the derived category of 
perverse sheaves}, p.~27--41, in ``$K$-theory, arithmetic and 
geometry'' ed. Yu. I. Manin,  Lect. Not. in Math. {\bf 1289} Springer-Verlag, 
1987.  

\bibitem[Bo]{nori:Bo} A. Borel, et al., \emph{Algebraic $D$-modules}, 
Perspectives in Mathematics, Vol. 2, Academic Press, 1987. 
  
\bibitem[KS]{nori:KS} M. Kashiwara, and P. Schapira, \emph{Sheaves on 
manifolds}, Grundlehren der mathematischen Wissenschaften {\bf 292}, 
Springer-Verlag, 1990.

\bibitem[GM]{nori:GM} M. Goresky, and R. Macpherson, \emph{Stratified Morse 
Theory}, Ergebnisse der Mathematik; 3. Folge, Band 14, Springer-Verlag 1988.

\bibitem[HL1]{nori:HL1} H.A. Hamm, and T. Le, \emph{Vanishing theorems for 
constructible sheaves, I}, J. Reine Angew. Math. {\bf 471} (1996), 115--138. 

\bibitem[HL2]{nori:HL2} H. Hamm, and  T. Le, \emph{Vanishing theorems for 
constructible sheaves, II}, Kodai Math. J. {\bf 21} (1998) no. 2, 208--247. 

\bibitem[Mac]{nori:Mac} S. Mac Lane,  \emph{Categories for the Working 
Mathematician}, Second edition, Graduate Texts in Math. {\bf 5}, 
Springer-Verlag, 1998.

\bibitem[Mil]{nori:Mil}J. Milnor, \emph{Morse theory}, 
Annals of Mathematics Studies, {\bf 51}, Princeton University Press, 1963. 

\bibitem[SGA4]{nori:SGA4} M. Artin, A. Grothendieck and J.L. Verdier, 
\emph{Theorie des Topos et Cohomologie Etale des
Schemas} SGA4, Tome 3, Springer LNM {\bf 305}, 1973. 

\bibitem[SGA4$\frac{1}{2}$]{nori:SGAfourhalf} P. Deligne, et al.,
\emph{Cohomologie etale} Springer LNM {\bf 569}, 1977.

\bibitem[Toh]{nori:Toh} A. Grothendieck, \emph{Sur Quelques Points D'algebre 
Homologique}, Tohoku Math. J., {\bf 9} (1957), 119--221.

\end{thebibliography}
\end{document}